\documentclass{amsart}
\usepackage{amsmath,amsfonts,amsthm,amssymb,indentfirst,epic,xurl,graphics,needspace}

\newtheorem{theorem}{Theorem}

\numberwithin{equation}{section}
\numberwithin{theorem}{section}

\renewcommand{\r}{\mathrm}

\begin{document}

\begin{center}
\texttt{Comments, corrections,
and related references welcomed, as always!}\\[.5em]
{\TeX}ed \today
\\[.5em]
\vspace{2em}
\end{center}

\title%
{An observation on factorizations of finite groups}
\thanks{%
Archived at \url{http://arxiv.org/abs/2608.08928}\,.
Readable
at \url{/~gbergman/papers/unpub}.
The latter version may be revised more frequently than the former.
}


\author{George M.\ Bergman}
\address{Department of Mathematics\\
University of California\\
Berkeley, CA 94720-3840, USA}
\email{gbergman@math.berkeley.edu}

\begin{abstract}
Given a finite group $G$ and a factorization of its order,
$|G|= a_1 a_2,$ we show that a sufficient condition for
there to exist subsets $A_1,\,A_2\subseteq G$ of
cardinalities $|A_1| = a_1,$ $|A_2| = a_2$ such that
$G = A_1\,A_2$ is that there exist a chain of subgroups
$\{e\} = G_0 < \dots < G_n = G$ such that $a_1$ is the
product of some subfamily of the indices $|G_i:G_{i-1}|$
(and hence $a_2$ is the product of the complementary subfamily).

Some further observations, and references to
related literature, are given.
\end{abstract}
\maketitle

\section{Main result}\label{S.main}

Given a finite group $G,$ and a factorization of its order,
\begin{equation}\begin{minipage}[c]{35pc}\label{d.a1ak}
$|G|\ =\ a_1\,\dots\,a_k\,,$
\end{minipage}\end{equation}
M.\,H.\,Hooshmand \cite{MHH_overflow}, \cite{MHH_factor}
raised the question of whether there
exist subsets $A_1,\dots,A_k$ of $G$ with cardinalities
$|A_i| = a_i$ $(i=1,\dots,k)$ such that
\begin{equation}\begin{minipage}[c]{35pc}\label{d.A1Ak}
$G\ =\ A_1\dots A_k$;
\end{minipage}\end{equation}
equivalently, in view of~\eqref{d.a1ak}, such that
the multiplication map $A_1\times\dots\times A_k\to G$
is one-to-one, equivalently, such that %
that same map is bijective.

For all $k > 2,$ examples are now known
of groups $G$ and $\!k\!$-fold factorizations of $|G|$
with all $a_i>1$ for which no such
decomposition of $G$ exists~\cite{GMB}, \cite{MK}, \cite{McC}.
But it is not known whether there are any such counterexamples
with~$k=2$~\cite[Question~20.37]{Kourovka21}.

We prove below a sufficient condition
(stated in the abstract above) on $G,$ $a_1,$ and $a_2$
for the existence of a decomposition $G=A_1\,A_2$ of the desired sort.

In dealing with products of finite sequences of subsets of a
group $G,$ we shall understand the product of the {\em empty}
sequence of subsets
(the sequence with no terms) to be the singleton~$\{e\}.$
It is easy to see that this convention
is compatible with composition of sequences.

Here is our result:

\begin{theorem}\label{T.main}
Let $G$ be a finite group, and
\begin{equation}\begin{minipage}[c]{35pc}\label{d.G0Gn}
$\{e\}\,=\,G_0\ <\ G_1\ <\ \dots\ <\ G_n\,=\,G$ \ $(n\geq 0)$
\end{minipage}\end{equation}
a chain of subgroups of $G.$

Then for any factorization $|G| = a_1\,a_2$ with $a_1$ and $a_2$
positive integers, a {\em sufficient}
condition for there to exist subsets $A_1,\,A_2\subseteq G$
having cardinalities $|A_1|=a_1,$ $|A_2|=a_2$ and satisfying
\begin{equation}\begin{minipage}[c]{35pc}\label{d.A1A2}
$G\ =\ A_1\,A_2$
\end{minipage}\end{equation}
is that there exist a subset $s_1\subseteq\{1,\dots,n\}$
such that $\prod_{i\in s_1}\,|G_i:\,G_{i-1}| = a_1,$
equivalently, such that, letting $s_2=\{1,\dots,n\}\setminus s_1,$
we have $\prod_{i\in s_2}\,|G_i:\,G_{i-1}| = a_2.$
\end{theorem}

\begin{proof}
Given a finite group $G,$ a chain~\eqref{d.G0Gn}, and
complementary sets $s_1,\ s_2\subseteq\{1,\dots,n\},$
we shall show by induction on $j$ that for each $j\leq n,$
\begin{equation}\begin{minipage}[c]{35pc}\label{d.A1jA2j}
there exists a factorization $G_j=A_{1,j}\ A_{2,j}$ such that \\
$|A_{1,j}|\ =\ \prod_{i\in s_1\cap\{1,\dots,j\}}\,|G_i:\,G_{i-1}|$ \ and
\ $|A_{2,j}|\ =\ \prod_{i\in s_2\cap\{1,\dots,j\}}\,|G_i:\,G_{i-1}|.$
\end{minipage}\end{equation}

This is trivially true for $j=0.$
(In that case the numerical products shown are over the empty set,
hence their values are both $1,$ and
$G_0=\{e\}$ indeed has the factorization $\{e\}=\{e\}\{e\}=
(\prod_\emptyset)\,(\prod_\emptyset).)$

So let $0\leq j<n,$ and assume inductively that
we have a factorization $G_j=A_{1,j}\ A_{2,j}$ as in~\eqref{d.A1jA2j}.

The integer $j\,{+}\,1$ occurs in exactly
one of the subsets $s_1,$ $s_2.$
Without loss of generality let us assume
\begin{equation}\begin{minipage}[c]{35pc}\label{d.i_in_s1}
$j\,{+}\,1\,\in\,s_1.$
\end{minipage}\end{equation}
(``Without loss of generality'' because our inductive
hypothesis and desired conclusion are symmetric
with respect to reversal of order of multiplication, which
interchanges the roles of $s_1$ and $s_2.)$
Assuming~\eqref{d.i_in_s1},
\begin{equation}\begin{minipage}[c]{35pc}\label{d.Cj+1}
let $C_{j+1}$ be any set of representatives
of the left cosets of $G_j$ in $G_{j+1},$ and let \\
$A_{1,j+1}\ =\ C_{j+1}\,A_{1,j},$ \ $A_{2,j+1}\ =\ A_{2,j}.$
\end{minipage}\end{equation}

Given that $A_{1,j}$ and $A_{2,j}$
have the cardinalities indicated in~\eqref{d.A1jA2j},
and assuming~\eqref{d.i_in_s1}, we see that
the new sets $A_{1,j+1}$ and $A_{2,j+1}$ will
have cardinalities as in the $j\,{+}\,1$ case of~\eqref{d.A1jA2j}.
Moreover, in view of~\eqref{d.Cj+1},
\begin{equation}\begin{minipage}[c]{35pc}\label{d.CiA1}
$A_{1,j+1}\,A_{2,j+1}\ =\ C_{j+1}\ A_{1,j}\ A_{2,j}\ =\ C_{j+1}\ G_j\
=\ G_{j+1}.$
\end{minipage}\end{equation}

So we have proved the $j\,{+}\,1$ case of~\eqref{d.A1jA2j}; so by
induction,~\eqref{d.A1jA2j} holds for $j=n,$ our desired conclusion.
\end{proof}

\section{Further observations and thoughts}\label{S.obs}

Given a group $G,$ the set of factorizations $|G| = a_1\,a_2$
that can be realized as in Theorem~\ref{T.main}
depends on what families of integers occur as the indices
$|G_i:\,G_{i-1}|$ for chains of subgroups~\eqref{d.G0Gn}.
If $G$ is solvable, then there exist such chains in
which every index is prime, from which it easily follows
that {\em every} factorization $|G| = a_1\,a_2$ can be so realized.

(In fact, the case of Theorem~\ref{T.main} in which all of the
indices $|G_i:\,G_{i-1}|$ are prime is \cite[Proposition~2]{BGV},
and the proof of that Proposition uses the same approach as
that of the above Theorem, and easily generalizes to give that Theorem.
Hence my decision not to publish the present note.)

There are also non-solvable groups for which that works.
For instance, one can get such a chain for the alternating group
$\r{Alt}(5),$ a simple group, by starting with such a chain for its
solvable subgroup $\r{Alt}(4),$ and then adding, at the end, the step
$\r{Alt}(4) < \r{Alt}(5),$ since $|\r{Alt}(5):\r{Alt}(4)| = 5,$
a prime.
On the other hand, this trick does not work for $\r{Alt}(6),$
since $|\r{Alt}(6):\r{Alt}(5)| = 6$ is not prime.
(And, in fact, this group has no subgroups of prime index
\cite[p.685, lines 5-7 after Proposition~2]{BGV}.)

Note, however, that the only $\!2\!$-term factorizations of
$|\r{Alt}(6)|=2^3\cdot 3^2\cdot 5 = 360$ that {\em cannot} be put
together from the prime indices in a chain for
$\r{Alt}(5),$ together with the index $6$ at the added step
$\r{Alt}(5) < \r{Alt}(6),$ are those in which all
occurrences of $2$ occur in one factor and all occurrences
of~$3$ in the other.
Up to order of factors there are just two
such factorizations: \ $2^3\cdot(3^2\cdot 5)$
and $3^2\cdot(2^3\cdot 5);$ and these can be realized using
two other chains of subgroups.
Namely, writing $\r{Syl}_p(\ )$ for ``a Sylow-$\!p\!$ subgroup of'',
we get these factorizations using the chains $\{e\} < \r{Syl}_2(G) < G$
and $\{e\} < \r{Syl}_3(G) < G$ respectively.
So every factorization of $|\r{Alt}(6)|$
{\em can} be realized as in Theorem~\ref{T.main} using one
or another chain of subgroups.

On the other hand, in \cite[p.685, middle paragraph]{BGV}
it is noted that the group $G = \r{SL}_2(8)$ of $2\times 2$
matrices of determinant $1$ over the field of $8$ elements has
properties which, from the point of view of the present note, say
that among the factorizations of $|G| = 504 = 2^3\cdot 3^2\cdot 7,$
neither of the two factorizations
\begin{equation}\begin{minipage}[c]{35pc}\label{d.21etc}
$21\cdot 24$ \ or \ $12\cdot 42$
\end{minipage}\end{equation}
can be realized as in Theorem~\ref{T.main}.
Indeed, every maximal proper subgroup of $G$ has one
of the orders $2^3\cdot 7$ or $2\cdot 3^2$ or $2\cdot 7$
\cite[p.2]{atlas}, hence these subgroups have indices
\begin{equation}\begin{minipage}[c]{35pc}\label{d.indices}
$3^2,$ \ $2^2{\cdot}7,$ \ $3^2{\cdot}2^2.$
\end{minipage}\end{equation}
Hence any chain~\eqref{d.G0Gn} for this group must have the index
at the last step divisible by one of the values~\eqref{d.indices}.
If that index is divisible by the first or last
of those terms, then in any factorization~\eqref{d.A1A2}
realized using such a chain, one factor must be
divisible by $3^2,$ which is not true of either of
the factorizations in~\eqref{d.21etc}, while if that index
is divisible by the middle term of~\eqref{d.indices}, then
in any factorization so realized,
one factor would have to be divisible by $2^2\cdot 7,$
which is again not true of either of the factorizations
of~\eqref{d.21etc}.

Nevertheless, it is shown in \cite{BGV} that the above group has
factorizations with the indicated cardinalities, using decompositions
$G = (H A'_1)(A'_2 K),$ where
$H$ and $K$ are subgroups of relatively prime orders,
and $A'_1$ and $A'_2$ (there denoted $A_0$ and $B_0)$
are subsets of small cardinalities whose existence the authors have
verified by computer searches (so their existence presumably does not
follow from any known general group-theoretic argument).

Curiously, although for the above group $G,$ not all factorizations
of its order can be realized as in Theorem~\ref{T.main},
all factorizations of the order of $G\times \mathbb{Z}_3$
can be so realized.
Indeed, given a factorization of the order of  that
product group, $2^3\cdot 3^3\cdot 7$
into two factors, one of those two factors must be divisible by $3^2.$
It is easy to see that on dividing an appropriate
one of those two factors by $3,$ we get a factorization of $|G|$
in which one of the factors is still divisible by $3^2.$
Hence {\em that} factorization of $|G|,$
not being one of~\eqref{d.21etc},
can be realized as in Theorem~\ref{T.main}.
Combining that factorization of $G$ with one of the factorizations
$\mathbb{Z}_3\times\{e\} \cong \mathbb{Z}_3 \cong
\{e\} \times \mathbb{Z}_3,$
we get a factorization of $G\times \mathbb{Z}_3$
realizing the given factorization of its order.

It is clear that if two groups $G_1$ and $G_2$ each have the
property that every factorization of its order can be realized
via Theorem~\ref{T.main}, then so does $G_1\times G_2$ (or more
generally, any extension of $G_1$ by $G_2).$
The above example shows, surprisingly, that the converse is not true.

In a different direction, one may ask whether an
approach similar to the one used in proving Theorem~\ref{T.main}
can be used to construct, from a chain~\eqref{d.G0Gn},
realizations of some class of
factorizations of $|G|$ with more than
two terms, e.g., $|G| = a_1\,a_2\,a_3.$
Note that in the construction of Theorem~\ref{T.main},
the choice, for each factor $|G_i:\,G_{i-1}|,$ of whether to
multiply by a set of that cardinality on the right or on the left
allows us to choose freely whether to bring it in as a factor
of $A_1$ or $A_2.$
If we are, rather, aiming for a factorization
$G = A_1\,A_2\,A_3,$ then we need to require
that in our expression for $G$ as a product of sets of
representatives of left and right cosets,
\begin{equation}\begin{minipage}[c]{35pc}\label{d.A2}
Of the sets whose product is to give the middle factor
$A_2,$ those that are adjoined on the left should arise from an
{\em initial} substring of $s_1,$ and those adjoined on the right
from an {\em initial} substring of~$s_2.$
\end{minipage}\end{equation}

Let us look at the case where $G$ is the
alternating group $\r{Alt}(4),$ for which it is known that
\begin{equation}\begin{minipage}[c]{35pc}\label{d.GMB_Prop}
\cite[Proposition~1.4]{GMB}
there is no decomposition $\r{Alt}(4) = A_1\,A_2\,A_3$ realizing
the factorization of $|\r{Alt}(4)| = 12$ as $2\cdot 3\cdot 2.$
\end{minipage}\end{equation}
and try to see what the obstruction is to obtaining
such a decomposition from a chain of subgroups.
The group $\r{Alt}(4)$
has, up to isomorphism, just two maximal chains of subgroups:
\begin{equation}\begin{minipage}[c]{35pc}\label{d.A4_long}
$\{e\}\ <\ \mathbb{Z}_2\ <\ %
\mathbb{Z}_2 \times \mathbb{Z}_2\ <\ \r{Alt}(4),$\ \ and
\end{minipage}\end{equation}
\begin{equation}\begin{minipage}[c]{35pc}\label{d.A4_short}
$\{e\}\ <\ \mathbb{Z}_3\ <\ \r{Alt}(4).$
\end{minipage}\end{equation}

If we put together cosets of the steps of the
chain~\eqref{d.A4_long} as in Theorem~\ref{T.main},
the step which brings in
a $\!3\!$-element factor is the last step;
but to make $|A_2|=3,$ by~\eqref{d.A2}, that step would have to be
the first that adjoins a factor on the side to which it is added,
hence both the other two factors must have been
adjoined on the other side, giving a decomposition corresponding
to the factorization $12  = 1 \cdot 3 \cdot 4$ or
$12 = 4 \cdot 3 \cdot 1.$
On the other hand, if we try to use~\eqref{d.A4_short},
there are only two factors, so we again cannot get
the desired $\!3\!$-factor expression.

(The above does not constitute a proof of~\eqref{d.GMB_Prop},
since~\eqref{d.GMB_Prop} is not restricted to decompositions arising
by the suggested generalization of the approach of Theorem~\ref{T.main}.
But it does illustrate the complications in trying to use
that approach to study group factorizations of lengths~$>2.)$

The above discussion focused on an obstruction to obtaining
from a chain of subgroups of a group $G$ a set-theoretic
representation of a given $\!k\!$-fold factorization of~$|G|.$
In \cite[Theorem~3.1]{McC} R.~McCulloch describes the conditions
under which that obstruction does not occur, and hence obtains
a sufficient condition for such factorizations of a group to exist.

Considerable information on results and questions on factorization
of groups prior to~\cite{MHH_overflow} is given in~\cite{MK}.

\section{Acknowledgements}\label{S.ackn}
I am indebted to Ryan McCulloch and Mikhail Kabenyuk for helpful
corrections and suggestions on earlier drafts of this note.

\end{document}